\documentclass[12pt]{amsart}

\usepackage{pdfsync}         
\usepackage[active]{srcltx}  

\usepackage{enumerate}
\usepackage{comment}
\usepackage{amssymb}

\usepackage[OT1]{fontenc}    %
\usepackage[usenames]{color}

\usepackage{palatino}
\usepackage[left=3cm,top=3.8cm,right=3cm]{geometry}        

\usepackage{xcolor}
\usepackage[pagebackref]{hyperref}
\hypersetup{
    colorlinks,
    linkcolor={red!60!black},
    citecolor={blue!60!black},
    urlcolor={blue!90!black}
}

\newcommand{\be}{\begin{equation}}
\newcommand{\ee}{\end{equation}}

\numberwithin{equation}{section}

\theoremstyle{plain}
\newtheorem{theorem}{Theorem}[section]
\newtheorem{corollary}[theorem]{Corollary}
\newtheorem{prop}[theorem]{Proposition}
\newtheorem{lemma}[theorem]{Lemma}

\theoremstyle{remark}

\theoremstyle{definition}
\newtheorem{definition}[theorem]{Definition}

\newcommand{\e}{\varepsilon}

\newcommand{\R}{\mathbb{R}}

\newcommand{\dist}{\mathrm{dist}}

\newcommand{\cN}{\mathcal{N}}
\newcommand{\cD}{\mathcal{D}}

\newcommand{\cR}{\mathcal{R}}

\DeclareMathOperator{\hdim}{dim_H}
\DeclareMathOperator{\pdim}{dim_P}

\DeclareMathOperator{\ubdim}{\overline{\dim}_B}

\DeclareMathOperator{\supp}{supp}

\newcommand{\wt}{\widetilde}
\newcommand{\al}{\alpha}
\newcommand{\de}{\delta}

\newcommand{\des}{\delta^s}
\newcommand{\dess}{\delta^{1-s}}
\newcommand{\om}{\omega}
\newcommand{\Om}{\Omega}
\newcommand{\lam}{\lambda}
\newcommand{\Lam}{\Lambda}

\title{On the packing dimension of Furstenberg sets}

\author{Pablo Shmerkin}
\address{Department of Mathematics and Statistics, Torcuato Di Tella University, and CONICET, Buenos Aires, Argentina}
\email{pshmerkin@utdt.edu}
\urladdr{http://www.utdt.edu/profesores/pshmerkin}
\thanks{P.S. was partially supported by Project PICT 2015-3675 (ANPCyT)}


\begin{document}

\begin{abstract}
We prove that if $\alpha\in (0,1/2]$, then the packing dimension of a set $E\subset\mathbb{R}^2$ for which there exists a set of lines of dimension $1$ intersecting $E$ in Hausdorff dimension $\ge \alpha$ is at least $1/2+\alpha+c(\alpha)$ for some $c(\alpha)>0$. In particular, this holds for $\alpha$-Furstenberg sets, that is, sets having intersection of Hausdorff dimension $\ge\alpha$ with at least one line in every direction. Together with an earlier result of T. Orponen, this provides an improvement for the packing dimension of $\alpha$-Furstenberg sets over the ``trivial'' estimate for all values of $\alpha\in (0,1)$. The proof extends to more general families of lines, and shows that the scales at which an $\alpha$-Furstenberg set resembles a set of dimension close to $1/2+\alpha$, if they exist, are rather sparse.
\end{abstract}

\maketitle

\section{Introduction and results}

\subsection{An improved lower bound on the packing dimension of Furstenberg sets}

In this paper we are concerned with Furstenberg sets and their variants. Let $\hdim$ denote Hausdorff dimension.
\begin{definition}
A set $E\subset\R^2$ is an $\alpha$-Furstenberg set if there exists a set $\Theta\subset S^1$ of positive measure, such that for all $\theta\in\Theta$ there is a line $\ell_\theta$ in direction $\theta$ such that
\[
\hdim(E\cap \ell_\theta) \ge \alpha.
\]
\end{definition}

An old problem motivated by work of Furstenberg \cite{Furstenberg70}, and first stated in print by Wolff \cite{Wolff99}, asks what is the smallest possible Hausdorff dimension of an $\alpha$-Furstenberg set. Wolff's paper contains a proof that if $E$ is an $\alpha$-Furstenberg set with $\alpha\in (0,1]$, then
\be \label{eq:easy-lower-bound}
\hdim(E) \ge \max(2\alpha,\alpha+1/2).
\ee
Note that the estimate is $\alpha+1/2$ is better for $\alpha\in (0,1/2)$ and both estimates coincide at $\alpha=1/2$. On the other hand, Wolff constructed an $\alpha$-Furstenberg set of Hausdorff dimension $\tfrac{1}{2}+\tfrac{3}{2}\al$. Thus \eqref{eq:easy-lower-bound} is sharp for $\al=1$, but it might not be otherwise.

It turned out to be quite hard to improve upon \eqref{eq:easy-lower-bound}. Katz and Tao \cite{KatzTao01} related the problem of improving the bound $\hdim(E)\ge 1$ for $1/2$-Furstenberg sets with other (at the time) outstanding problems in geometric measure theory. Shortly after, Bourgain \cite{Bourgain03} made dramatic progress on one of these related problems, the discretized sum-product problem. Combining the results of \cite{KatzTao01} and \cite{Bourgain03} it follows that $\hdim(E)\ge 1+c$ for all $1/2$-Furstenberg sets, where $c>0$ is a small universal constant. In fact, the proof shows that the same holds for $\alpha$-Furstenberg sets if $\alpha$ is close enough to $1/2$ (possibly smaller). This remains the only improvement to \eqref{eq:easy-lower-bound} to date.

In \cite{Orponen20}, however, Orponen achieved an improvement over \eqref{eq:easy-lower-bound} in the range $\alpha\in (1/2,1)$, with the caveat that he considered the \emph{packing} dimension $\pdim(E)$ of $E$, rather than Hausdorff dimension. That is, he proved that for each $\alpha\in (1/2,1)$ there is a small number $c(\alpha)>0$ such that
\[
\pdim(E) \ge 2\alpha+c
\]
for all $\alpha$-Furstenberg sets $E$. Since (upper) box-counting dimension is at least as large as packing dimension, this estimate holds for it as well. See \cite[Chapters 2 and 3]{Falconer14} for an introduction to the notions of dimension considered in this paper.

Orponen's result left open the question of whether such an improvement is also possible, at least for packing dimension, also in the range $\alpha\in (0,1/2)$.  Our main result provides such an improvement:

\begin{theorem} \label{thm:main}
For $\alpha \in (0,1/2]$ there exists a constant $c=c(\alpha)>0$ such that
\[
\pdim(E) \ge \frac{1}{2}+\alpha +c
\]
for all $\alpha$-Furstenberg sets $E\subset\R^2$.
\end{theorem}

We can in fact prove a more general result. In \cite{MolterRela12}, Molter and Rela generalized the notion of Furstenberg set by considering also fractal sets of directions. In several later papers such as \cite{HKM19, HSY20} this was expanded further to arbitrary families of lines (not necessarily pointing in different directions). Note that the Grassmanian $A(2,1)$ of affine lines in the plane is a $2$-dimensional manifold, so one can speak of the Hausdorff dimension of an arbitrary set of lines.  We say that $E\subset\R^2$ is an $(\alpha,\beta)$-Furstenberg set if
\[
\hdim\{ \ell\in A(2,1):\hdim(E\cap \ell) \ge \alpha\}\ge \beta.
\]
Clearly, a classical $\alpha$-Furstenberg set is an $(\alpha,1)$-Furstenberg set. Let $s(\al,\beta)$ be the infimum of the Hausdorff dimensions of $(\al,\beta)$-Furstenberg sets. In the range $\al\in (0,1/2]$, the following bounds are known:
\[
s(\al,\beta) \ge \left\{
\begin{array}{cccl}
  \al+\min(\al,\beta) & \text{if} & \beta \le 2\alpha & \text{(\cite{LutzStull17, HSY20})} \\
  2\al+ c(\al) & \text{if} & \beta \in (2\al-c(\al),2\al+c(\al)) & \text{(\cite{HSY20})} \\
  \al + \beta/2 & \text{if} & \beta\in (2\al, 1] & \text{(\cite{MolterRela12,Hera19})} \\
\end{array}
\right..
\]
We are able to improve upon the $\al+\beta/2$ bound in the range $\beta>2\al$, with the caveat that we consider the packing (or upper box-counting) dimension of the $(\al,\beta)$-Furstenberg set. Let $\cN_\de(A)$ denote the smallest number of balls of radius $\de$ needed to cover $A\subset\R^d$.
\begin{theorem} \label{thm:main-alpha-beta}
Given $\alpha \in (0,1/2]$ and $\beta \in (2\alpha,2]$ there is $c=c(\alpha,\beta)>0$ (depending continuously on the parameters) such that the following holds.

Let $E$ be an $(\alpha,\beta)$-Furstenberg set and fix $s\in [1/2,1)$. Then for all small enough $\de$ (depending on $E$ and $s$), either
\be  \label{eq:main-scale-dichotomy}
\cN_\de(E) \ge \de^{-\alpha-\beta/2-(1-s)c} \quad\text{or}\quad \cN_{\des}(E) \ge (\des)^{-\alpha-\beta/2-(1-s)c}.
\ee
\end{theorem}

We make a few remarks on this statement. Firstly, by taking $s=1/2$, it follows immediately that $\ubdim(E)\ge \al+\beta/2+c'(\al,\beta)$, with $c'(\al,\beta)=c(\al,\beta)/2$. By a standard argument (see \cite[Section 2]{Orponen20}), the same is true for packing dimension:
\begin{corollary} \label{cor:main-alpha-beta}
In the context of Theorem \ref{thm:main-alpha-beta},
\[
\ubdim(E) \ge \pdim(E) \ge \alpha+\frac{\beta}{2}+c'(\al,\beta).
\]
\end{corollary}
Of course, this implies Theorem \ref{thm:main} as a special case.

Secondly, although our method falls short of proving that
\be \label{eq:E-large-scale-delta}
\cN_\delta(E)\gtrsim \delta^{-\al+\beta/2+c}
\ee
for all small $\de$ (which would imply that $\hdim(E)\ge  \al+\beta/2+c$), Theorem \ref{thm:main-alpha-beta} shows that the scales $\de$ such that \eqref{eq:E-large-scale-delta} fails are quite sparse (while $\ubdim(E)\ge \al+\beta/2+c$ only requires that \eqref{eq:E-large-scale-delta} holds for arbitrarily small $\de$) . Namely, given $\e>0$, for any small $\de_0$, if \eqref{eq:E-large-scale-delta} fails for $\de_0$, then it must hold for all $\de\in [\de_0^{1/2},\de_0^2] \setminus [\de_0^{1-\e},\de_0^{1+\e}]$ (with $c=c(\al,\beta,\e)$).

\subsection{Sketch of the proof of Theorem \ref{thm:main-alpha-beta}}

The methods of \cite{KatzTao01, Orponen20, HSY20} all consist in starting with a counter-assumption that the dimension is very close to the one we want to beat, and deduce (after some substantial processing and transformation of the set) that there must be some underlying ``product structure''. The proof is concluded by using Bourgain's discretized sum-product or projection theorems \cite{Bourgain03,Bourgain10} to parlay the product structure into a small dimensional gain. We use a similar approach: under the counter-assumption that \eqref{eq:main-scale-dichotomy} fails, we show that $E$ (after refinement) must have a quite rigid structure at scales $\des$ and $\de$. This structure (which is not quite of product type) enables the application of Bourgain's projection theorem. The reduction to Bourgain's projection theorem is different from, and perhaps simpler than, those of \cite{KatzTao01, Orponen20, HSY20} .

After some standard reductions detailed in \S\ref{subsec:reductions}, at scale $\de$ we replace the $(\al,\beta)$-Furstenberg set by a union $E'=\cup_{\om\in\Om} E_\om$, where $\Om$ is a $\de$-separated set of (parameters of) lines satisfying a $\beta$-dimensional Frostman assumptions, and $E_\om$ are subsets of the corresponding lines satisfying a uniform $\alpha$-dimensional Frostman bound.

We start with a counter-assumption that $\cN_{\des}(E) \le (\des)^{-\al-\beta/2-\eta}$ for a small $\eta>0$. Using a variant of the estimate that provides the lower bound $\al+\beta/2$ for the dimension of $(\alpha,\beta)$-Furstenberg sets, given in Lemma \ref{lem:alpha+gamma bound}, we show that the sets $E_\om$ must cluster nearly as much as possible at scale $\des$. Namely, we cover $E'$ (after refinement) by $\approx$ $(\des)^{-\beta}$ sets $F_\lam$, contained in rectangles of size $\des \times 1$ whose long directions are $\des$-separated, and such that $\cN_{\des}(F_\lam)$ is not much larger than $(\des)^{-\al}$. Moreover, each $F_\lam$ contains $\approx (\dess)^{\beta}$ of the sets $E_\om$. This analysis is carried out in \S\ref{subsec:scale-des}.

In \S\ref{subsec:analysis-F-lam}, we analyze a fixed set $F_\lam$ at scale $\de$. We show that it can be essentially decomposed into $\approx (\des)^{-\al}$ squares $Q_j$ of side length $\des$ such that ``most'' of the set $E_\om\subset F_\lam$ have large intersection with $Q_j$. By the assumption $s\ge 1/2$, the lines containing the sets $E_\om$ are nearly parallel inside $Q_j$. This enables the application of Bourgain's discretized projection theorem (recalled as Theorem \ref{thm:projection-Bourgain} below) to show that for most such $Q_j$ we have a box-counting gain:
\[
\cN_{\de}(E\cap Q_j) \ge (\dess)^{-\al-\beta/2-\zeta},
\]
where $\zeta>0$ depends only on $\al$ and $\beta$. Finally, a double counting of the squares $Q_j$ presented in \S\ref{subsec:conclusion-proof} allows us to conclude the proof of Theorem \ref{cor:main-alpha-beta}.

\section{Preliminaries}

\subsection{Notation}

We use the notation $A\lesssim B$ to mean $A\le C B$ for some constant $C>0$. If $C$ depends on certain parameters $\text{par}$, we write $A\lesssim_{\text{par}} B$. We write $A\gtrsim B$ for $B\lesssim A$ and $A\sim B$ for $A\lesssim B\lesssim A$, and likewise with parameters. Sometimes we write $A=O(B)$ to mean $A\lesssim B$.

We denote (open) balls of centre $x$ and radius $r$ in $\R^d$ by $B^d(x,r)$. We often omit the superscript $d$. We write $B^d=B^d(0,1)$ for short. The notation $B_r$ denotes an arbitrary ball of radius $r$ (with the ambient dimension fixed but implicit).

Given $\om=(a,b)$, we let $L_\om(t)=at+b$. Abusing notation slightly, we will also denote the line $\{ y=at+b\}$ by $L_\om$.

Recall that $\cN_\de(A)$ denotes the $\de$-covering number of $A$.

\subsection{Discretization}

As usual in this circle of problems, in the proof of Theorem \ref{thm:main-alpha-beta}, we will work with certain discretized versions of Furstenberg sets, which we now introduce.

\begin{definition}
A probability measure $\mu$ on $B^1$ is a $(\delta,\alpha,K)$-measure if
\[
\mu(B_r)\le K r^\alpha \quad\text{for all } r\in [\delta,1].
\]
\end{definition}

Note that a $(\delta,\alpha,K)$-measure is automatically a $(\delta',\alpha',K')$-measure whenever $\delta'\in [\delta,1]$, $\alpha'\le \alpha$ and $K'\ge K$. Also, if $\mu$ is a $(\delta,\alpha,K)$ measure and $\mu(A)\ge c>0$, then $\cN_\de(A)\gtrsim_{c,K,\alpha} \delta^{-\alpha}$. We will repeatedly use these facts without further mention. For us, the crucial parameters are the scale $\delta$ and the exponent $\alpha$; the value of the constant $K$ will not concern us too much.

\begin{definition}
We say that $(\nu_\om)_{\om\in\Om}$ is a $(\delta,\alpha,K)$-Furstenberg tuple if:
\begin{itemize}
  \item $\Om\subset B^2$ is a $\delta$-separated set,
  \item For each $\om$, there is a $(\de,\al,K)$-measure $\nu'_\om$ such that $\nu_\om= L'_\om \nu'_\om$, where $L'_\om(x)=(x,L_\om(x))$.
\end{itemize}
\end{definition}

We will obtain a Furstenberg tuple by starting with a Furstenberg set $E$, taking any $\delta$-separated set $\Omega$ of the (parameters of the) lines satisfying \eqref{eq:positive-content}, and using Frostman's Lemma to obtain a suitable measure $\nu_\om$ on $E\cap L_\om$ for each $\om\in\Om$. See \S\ref{subsec:reductions} for details.

In the following two lemmas, we fix a $(\delta,\alpha,K)$-Furstenberg tuple $(\nu_\om)_{\om\in\Om}$, and write $E_\om=\supp(\nu_\om)$ and $E=\cup_{\om\in\Om} E_\om$.

Our first lemma is a small box-counting variant of the argument giving the lower bound $1/2+\alpha$ for $\alpha$-Furstenberg sets.
\begin{lemma} \label{lem:alpha+gamma bound}
Suppose $\alpha>0$ and $\cN_\delta(E_\om)\ge M$ for all $\om\in\Omega$. Then
\[
(\cN_\delta(E))^2 \gtrsim_{K,\alpha} |\Omega| M\delta^{-\alpha}.
\]
\end{lemma}
\begin{proof}
Pick $c=c_{K,\alpha}$ such that $K(3c)^\alpha\le 1/2$; hence $\nu_\om(B_{3c})\le 1/2$ for all $\om\in\Om$. By assumption, for each $\om\in\Om$ there is a segment of length $c$, say $I_\om\subset L_\om$, such that $\cN_\delta(E_\om\cap I_\om)\gtrsim_{c} M$. Since $\nu_\om(3I_\om)\le 1/2$, it follows from the Frostman condition $\nu_\om(B_r)\le K r^\alpha$ that $\cN_\delta(E_\om\setminus 3I_\om)\gtrsim_K \delta^{-\alpha}$.

Let $\ell_{x,y}$ denote the line through $x$ and $y$. If $x$ is at distance $\le \delta$ from $I_\om$ and $y$ is at distance $\le\delta$ from $L_\om \setminus 3 I_\om$, then the distance on $A(2,1)$ between $\ell_{x,y}$ and $L_\om$ is $\lesssim_K \de$.  This implies that if $d(\om_1,\om_2)\ge C\delta$, for a sufficiently large constant $C=C_K$, then the sets $(E_{\om_i}\cap I_{\om_i}) \times (E_{\om_i}\setminus 3 I_{\om_i})$, $i=1,2$, are $\delta$-separated. Thus
\[
(\cN_\delta(E))^2\sim \cN_\delta(E\times E) \ge \cN_\de\left(\bigcup_{\om\in\Om} E_\om\times E_\om\right) \gtrsim_{K,\alpha} |\Omega|M\delta^{-\alpha},
\]
giving the claim.
\end{proof}

\begin{lemma} \label{lem:fan}
If there is a point $x$ such that $\dist(x,L_\om)\le C\delta$ for all $\om\in\Om$, then
\[
\cN_\delta(E) \gtrsim_{C,K} |\Om|\delta^{-\alpha}.
\]
\end{lemma}
\begin{proof}
Pick $c=c_{K,\alpha}$ small enough that $\nu_\om(B_c)\le 1/2$ for all $\om\in\Om$. By elementary geometry, if $d(\om_1,\om_2)\ge C'\delta$ for a sufficiently large constant $C'=C'_{K,c}$, the sets $E_{\om_i}\setminus B(x,c)$ are $\delta$-separated. Since $\cN_\delta(E_\om\setminus B(x,c))\gtrsim_K \delta^{-\alpha}$, the lemma follows.
\end{proof}

\subsection{Bourgain's projection theorem}

We state Bourgain's discretized projection theorem for later reference. We borrow the version from \cite[Theorem 1]{He18}, in the special case of projections from the plane to the real line. Given $\theta\in S^1$, let $P_\theta:\R^2\to\R$, $x\mapsto \langle x,\theta\rangle$ denote projection in direction $\theta$.
\begin{theorem} \label{thm:projection-Bourgain}
Given $0<u<2$ and $0<t<1$ there exist $\e,\zeta>0$ such that the following holds for all sufficiently small $\de$. Let $X\subset B^2$, and let $\rho$ be a probability measure on $S^1$ such that the following hold:
\begin{align*}
\cN_\delta(X) &\ge \delta^{\e-u},\\
\cN_\delta(X\cap B(x,r)) &\le \delta^{-\e} r^t \cN_\delta(X) \quad\text{for all } r\in [\delta,1], x\in B^2,\\
\rho(B_r) &\le \delta^{-\e} r^t \quad\text{for all } r\in [\de,1].
\end{align*}
Then there is a set $\Theta\subset S^1$ with $\rho(\Theta)\le \delta^\e$ such that if $\theta\in S^1\setminus \Theta$ and  $X'\subset X$ satisfies $\cN_\delta(X') \ge \de^\e\cN_\delta(X)$, then
\[
\cN_\delta(P_\theta X') \ge \de^{-\frac{u}{2}-\zeta}.
\]
\end{theorem}

It follows from the robustness of the hypotheses that $\eta,\zeta$ can be taken continuous in $t,u$.

\section{Proof of Theorem \ref{thm:main-alpha-beta}}

\subsection{Initial reductions} \label{subsec:reductions}

Recall that the (spherical) $\gamma$-dimensional Hausdorff content is defined as
\[
\mathcal{H}_\infty^\gamma(A) = \inf\left\{ \sum_i r_i^\gamma: A\subset \cup_i B(x_i,r_i)\right \}.
\]
Standard, easy to verify properties are that $\mathcal{H}_\infty^\gamma$ is countably subadditive, and $\mathcal{H}_\infty^\gamma(A)>0$ if and only if $\mathcal{H}^\gamma(A)>0$, with the latter denoting standard Hausdorff measure.

Let $E$ be an $(\al,\beta)$-Furstenberg set and let $\mathcal{L}$ denote the set of lines $\ell$ such that $\hdim(E\cap\ell)\ge \al$; then $\hdim(\mathcal{L})\ge \beta$ by definition. By passing to a subset of $\mathcal{L}$, we may assume that the slopes of all lines in $\mathcal{L}$ either are $\ge 1$ or $\le 1$ in absolute value. By rotating $E$ if needed, we may assume that all slopes are $\le 1$ in absolute value. We can then identify $\mathcal{L}$ with a subset of $[-1,1]\times\R$ via $\om\leftrightarrow L_\om$; this is a smooth identification, in particular it preserves Hausdorff dimension.

Fix $\e>0$. By the countable stability of Hausdorff dimension and the countable subadditivity of Hausdorff content, we may find $b_0\in\R$ such that
\[
\mathcal{H}_\infty^{\beta-\e}(\{\om:L_\om\in\mathcal{L}\}\cap B^2((0,b_0),1)) > 0.
\]
By translating $E$, we may assume that $b_0=0$.  Now, since $\mathcal{H}_\infty^{\al-\e}(E\cap L_\om)>0$ for all $L_\om\in\mathcal{L}$, the countable subadditivity of $\mathcal{H}_\infty^{\beta-\e}$ shows that there exists $c>0$ such that
\[
\mathcal{H}_\infty^{\beta-\e}\{\om\in B^2: \mathcal{H}_\infty^{\al-\e}(E\cap L_\om)\ge c \} > 0.
\]
Since $\e>0$ is arbitrary, this shows that in order to establish Theorem \ref{thm:main-alpha-beta}, we may assume that $\mathcal{L}\subset \{ L_\om: \om\in\wt{\Om}\}$ for some $\wt{\Om}\subset B^2$, $\mathcal{H}_\infty^\beta(\wt{\Om})>0$, and
\be \label{eq:positive-content}
\mathcal{H}_\infty^\alpha(E\cap L_\om) \ge c > 0
\ee
for all $\om\in\wt{\Om}$.

Fix a small $\de>0$ and $s\in [1/2,1)$. In the rest of the paper, any implicit constants are allowed to depend on $K,c,\alpha$ and $\beta$, but they must always be independent of $\de$.

F\"{a}ssler and Orponen \cite[Proposition A.1]{FasslerOrponen14} proved a discretized version of Frostman's Lemma (they stated it in $\R^3$ but the same proof works in any dimension). Applying this result to the set $\wt{\Om}$, we obtain a $\delta$-separated set $\Om\subset \wt{\Om}$ such that
\eqref{eq:positive-content} holds for all $\om\in \Om$ and, moreover, $\Om$ satisfies the Frostman condition
\be \label{eq:Om-Frostman}
|\Om\cap B_r|\lesssim r^\beta |\Om|\quad\text{for all }r\in [\de,1].
\ee
Since $\Om$ is nonempty we get that, in particular, $|\Om|\gtrsim \de^{-\beta}$.

Let $E^{(\delta)}$ denote the $\delta$-neighborhood of $E$. Since $\cN_\de(E^{(\de)})\sim \cN_\de(E)$ and $\cN_{\des}(E^{(\de)})\sim \cN_{\des}(E)$, we may replace $E$ by $E^{(\delta)}$; in particular, this allows to assume that $E$ is Borel. Applying Frostman's Lemma to each $E\cap L_\om$, we obtain a $(\delta,\alpha,K)$-Furstenberg tuple $(\nu_\om)_{\om\in\Om}$ where $K=K(c)$. We recall that in Frostman's Lemma the constant factor depends only on the Hausdorff content on the set, and therefore \eqref{eq:positive-content} ensures that a suitable uniform $K$ can be found. Note that
\[
E\supset \bigcup_{\om\in\Om} \supp(\nu_\om) =: \bigcup_{\om\in\Om} E_\om.
\]
and hence it suffices to estimate $\cN_\de$, $\cN_{\des}$ for the union on the right-hand side.

\subsection{Analysis at scale $\des$}
\label{subsec:scale-des}

Assume that
\be  \label{eq:bound-scale-sqr-delta}
\cN_{\des}(E) \le (\des)^{-\alpha-\beta/2-\eta}
\ee
We first show that this counter-assumption forces a rather rigid structure on $E$ at scale $\des$ (at least after refinement of $E$).

Let $\Lambda$ be a maximal $\des$-separated subset of $\Om$. Note that $|\Lambda|\gtrsim (\des)^{-\beta}$ by \eqref{eq:Om-Frostman}.
\begin{lemma} \label{lem:upper-bound-Lambda}
\[
|\Lambda| \lesssim (\des)^{-\beta-2\eta}.
\]
\end{lemma}
\begin{proof}
Note that $( \nu_\om )_{\om\in\Lambda}$ is a $(\des,\alpha,K)$-Furstenberg tuple; in particular, $\cN_{\des}(E_\om)\gtrsim (\des)^{-\al}$. The claim follows from Lemma \ref{lem:alpha+gamma bound} and our assumption \eqref{eq:bound-scale-sqr-delta}.
\end{proof}

For each $\lambda\in\Lambda$, let
\[
\Om_\lam = \Om\cap B^2(\lam,\des).
\]
Note that the sets $\Om_\lam$ have bounded overlapping and cover $\Om$.

\begin{lemma} \label{lem:many-rich-deh-tubes}
There is a set $\Lam'\subset\Lam$ such that
\[
|\Lambda'| \gtrsim (\des)^{-\beta}
\]
and
\be \label{eq:Lambda'}
|\Om_\lam| \gtrsim (\des)^{\beta+2\eta}|\Om|
\ee
for all $\lam\in\Lam'$.
\end{lemma}
\begin{proof}
Let
\[
\Lambda' =\{\lam\in\Lam: |\Om_\lam| \ge C^{-1}(\des)^{\beta+2\eta}|\Om|\}.
\]
It follows from Lemma \ref{lem:upper-bound-Lambda} that if $C$ is taken large enough then
\[
\sum_{\lam\in\Lam\setminus\Lam'} |\Om_\lam| \le |\Om|/2.
\]
Therefore, using \eqref{eq:Om-Frostman},
\[
|\Om|/2\le \sum_{\lam\in\Lam'} |\Om_\lam| \lesssim |\Lam'| (\des)^{\beta}|\Om|,
\]
as claimed.
\end{proof}

Fix $\lam\in\Lam'$. Let
\begin{align*}
\rho_\lam&= \frac{1}{|\Om_\lam|} \sum_{\om\in \Om_\lam} \nu_\om,\\
F_\lam &= \bigcup_{\om\in\Om_\lam} E_\om =\supp(\rho_\lam).
\end{align*}
Let us also denote by $\wt{\rho}_\lam$, $\wt{F}_\lam$ the orthogonal projections of $\rho_\lam, F_\lam$ onto the line $L_\lam$.

\begin{lemma} \label{lem:Frostman-average}
\[
\wt{\rho}_\lam(B_r) \lesssim r^\alpha \quad \text{for } r\in [\de,1].
\]
\end{lemma}
\begin{proof}
Since each $L_\om, \om\in\Om_\lam$ makes an angle $\lesssim \des\ll 1$ with $L_\lam$, the orthogonal projection of $\nu_\om$ to $L_\lam$ satisfies the same Frostman condition as $\nu_\om$ (other than the constant), and hence so does an average of such projections.
\end{proof}

\begin{lemma} \label{lem:deh-rectangles-small-projection}
If $C=C_{K,\alpha}$ is large enough,
\begin{equation} \label{eq:deh-rectangles-small-projection}
\left|\{\lam\in\Lam': \cN_{\des}(\wt{F}_\lam) \ge C(\des)^{-\alpha-2\eta}\}\right| \le |\Lam'|/2.
\end{equation}
\end{lemma}
\begin{proof}
Since $\wt{F}_\lam$ and $F_\lambda$ are $O(\des)$-close in the Hausdorff metric, the assumption \eqref{eq:bound-scale-sqr-delta} yields
\[
\cN_{\des}\left(\bigcup_{\lam\in\Lam'}\wt{F}_\lam\right) \sim \cN_{\des}\left(\bigcup_{\lam\in\Lam'} F_\lam\right) \le \cN_{\des}(E) \le (\des)^{-\alpha-\beta/2-\eta} .
\]
Let $\Lambda''\subset\Lam'$ be the set on the left-hand side of \eqref{eq:deh-rectangles-small-projection}. By Lemma \ref{lem:Frostman-average},  $(\wt{\rho}_\lam)_{\lam\in\Lambda''}$ is a $(\des,\alpha,O_K(1))$-Furstenberg tuple. It follows from Lemma \ref{lem:alpha+gamma bound} that $|\Lambda''|\le c (\des)^{-\beta}$, where $c$ can be made arbitrarily small by taking $C$ large. The conclusion now follows from Lemma \ref{lem:many-rich-deh-tubes}.
\end{proof}
In light of this lemma, by passing to a subset of $\Lambda'$ that still satisfies the conclusions of Lemma \ref{lem:many-rich-deh-tubes}, we may assume from now on that
\begin{equation}  \label{eq:assumption-deh-rectangles-small-projection}
\cN_{\des}(\wt{F}_\lam) \lesssim (\des)^{-\alpha-2\eta} \quad\text{for all }\lam\in\Lambda'.
\end{equation}

\subsection{Analysis of a fixed $F_\lam$ at scale $\de$}
\label{subsec:analysis-F-lam}

Fix $\lam\in\Lam'$ for the rest of this section. Cover $L_\lam\cap B^2$ by intervals of length $2\des$. As above, let $\wt{F}_\lam$ be the orthogonal projection of $F_\lam$ onto the line $L_\lam$. Let us denote the subset of $(2\des)$-intervals that intersect $\wt{F}_\lam$ by $(I_j)_{j=1}^N$.  By \eqref{eq:assumption-deh-rectangles-small-projection}, we know that
\be \label{eq:N-upper-bound}
N \lesssim (\des)^{-\alpha-2\eta}.
\ee
Denote by $Q_j$ the square of side length $2\des$ whose orthogonal projection onto $L_\lam$ is $I_j$, and such that $I_j$ divides it into equal pieces.

\begin{lemma} \label{lem:J}
There exists a set $J\subset \{1,\ldots,N\}$ such that
\[
\sum_{j\in J} \wt{\rho}_\lam(I_j) \ge 1/2,
\]
and
\be \label{eq:om-j-large}
|\Om_\lam^{j}|\gtrsim (\des)^{2\eta} |\Om_\lam|
\ee
for all $j\in J$, where
\be \label{eq:om-large-int-with-square}
\Om_\lam^{j} = \{ \om\in\Om_\lam: \cN_\de(E_\om\cap Q_j)   \gtrsim \de^{2\eta s} (\dess)^{-\al} \}.
\ee
\end{lemma}
\begin{proof}
By Lemma \ref{lem:Frostman-average}, $\wt{\rho}_\lam(I_j)\lesssim (\des)^\alpha$ for all $j$. Let
\[
J = \{j\in\{1,\ldots,N\}:\wt{\rho}_\lam(I_j)\ge C^{-1}(\des)^{\alpha+2\eta}\}.
\]
Notice that
\[
\sum_{j\in\{1,\ldots,N\}\setminus J} \wt{\rho}_\lam(I_j) \lesssim  C^{-1}N(\des)^{\alpha+2\eta},
\]
and thus if $C$ is taken large enough, the first claim follows from \eqref{eq:N-upper-bound}.

We now show the second claim. It follows from the definitions that
\[
\wt{\rho}_\lam(I_j) = \frac{1}{|\Om_\lam|}\sum_{\om\in\Om_\lam} \nu_\om(Q_j).
\]
Splitting the sum depending on whether $\nu_\om(Q_j)\ge c(\des)^{\alpha+2\eta}$ or not, where $c$ is a suitably small constant, we deduce that
\[
|\{ \om\in\Om_\lam: \nu_\om(Q_j)\ge c(\des)^{\alpha+2\eta} \}| \gtrsim (\des)^{2\eta}|\Om_\lam|.
\]
The second claim now follows from the Frostman condition of $\nu_\om$.
\end{proof}

We now introduce the projections that will allow us to apply Bourgain's projection theorem in our setting, and state the key lemma that relates the box-counting numbers $\cN_\de(E\cap Q_j)$ to such projections.

Let
\[
\pi_t(\om) = L_\om(t)\quad\text{or}\quad \pi_t(a,b)=at+b.
\]
\begin{lemma} \label{lem:box-counting-from-projection}
Fix $j\in J$. Then, for any $t\in I_j$,
\[
\cN_\de(E\cap Q_j) \gtrsim   \de^{2\eta s} (\dess)^{-\al}  \, \cN_{\de}\left(\pi_t \Om_\lam^{(j)}\right).
\]
\end{lemma}
\begin{proof}
Fix $t\in I_j$. Since all $L_\om$, $\om\in\Om_\lam$ make an angle $\le 2 \des$ with each other and $Q_j$ has diameter $<4 \des$, the segments $L_\om\cap Q_j, L_{\om'}\cap Q_j$ are $(2\de)$-separated provided  $|\pi_t(\om)-\pi_t(\om')| \ge 10\de$. This is the point where we use that $s\ge 1/2$.

We can thus take a subset of $\Om_\lam^{(j)}$ of size $\sim \cN_{\de}(\pi_t \Om_\lam^{(j)})$ such that the corresponding segments $L_\om\cap Q_j$ are $(2\de)$-separated. The claim now follows from \eqref{eq:om-large-int-with-square}.
\end{proof}

We are now able to apply Bourgain's projection theorem to obtain a gain in the size of $\cN_\de(E\cap Q_j)$ for most $j$.
\begin{prop} \label{prop:application-bourgain-proj-thm}
If $\zeta,\eta$ are sufficiently small in terms of $\al,\beta$ (with the threshold depending in a continuous manner on $\al,\beta$), then for $\de$ small in terms of all the previous parameters,
\[
\cN_\de(E \cap Q_j) \gtrsim (\dess)^{-\al-\beta/2-\zeta}
\]
on a set of $j$ such that $\sum_j \wt{\rho}_\lam(I_j)\ge 1/4$.
\end{prop}
\begin{proof}
In light of Lemmas \ref{lem:J} and \ref{lem:box-counting-from-projection}, it is enough to show that there is $\zeta=\zeta(\alpha,\beta)>0$ such that, provided $\eta=\eta(\alpha,\beta)>0$ is sufficiently small, then
\be \label{proj-to-prove}
\cN_{\de}\left(\pi_t\Om_\lam^{(j)}\right) \ge (\dess)^{-\beta/2-\zeta}
\ee
for all $t$ in a set of $\wt{\rho}_\lam$-measure $\ge 3/4$. (Making $\eta$ small enough that $2\eta \le \zeta/2$, this will give the claim with $\zeta/2$ in place of $\zeta$.) We will see that this is a consequence of Theorem \ref{thm:projection-Bourgain}. To begin, recall from Lemma \ref{lem:Frostman-average} that $\wt{\rho}_\lam(B_r)\lesssim r^\alpha$ for $r\in [\delta,1]$.

We now rescale and translate $\Om_\lam$ to bring it into the setting of Theorem \ref{thm:projection-Bourgain}: let
\[
X_\lam = \de^{-s}(\Om_\lam-\lam) = \{ \de^{-s}(\om-\lam):\om\in\Om_\lam\}.
\]
Note that $X_\lam\subset B^2$ and $X_\lam$ is $(\dess)$-separated. It follows from \eqref{eq:Om-Frostman} and \eqref{eq:Lambda'} that, for $r\in [\dess,1]$,
\[
|X_\lam\cap B_r| = |\Om_\lam\cap B_{\des r}| \lesssim (\des r)^{\beta}|\Om| \lesssim \de^{-2\eta s} r^\beta |X_\lam|.
\]
In particular, taking $r=\dess$,
\[
|X_\lam| \gtrsim \de^{2\eta s} (\dess)^{-\beta}.
\]
Thus (provided $\eta$ is small enough in terms of $\al$ and $\beta$ only) $X_\lam$ and $\wt{\rho}_\lam$ satisfy the assumptions of Theorem \ref{thm:projection-Bourgain} (with $t=\alpha$, $u=\beta$ and scale $\dess$ in place of $\de$). Now the maps $\pi_t$ are, up to bounded rescaling, a smooth reparametrization of the family of orthogonal projections (given by $\theta\mapsto t=\tan(\theta)$).

Applying Theorem \ref{thm:projection-Bourgain}, at scale $\dess$, to the rescaled and translated sets $X'_j=\de^{-s}(\Om_\lam^{j}-\lam)$ (which is permissible by \eqref{eq:om-j-large}), we conclude that there are $\zeta=\zeta(\alpha,\beta)>0$, $\eta=\eta(\alpha,\beta)$ (depending continuously on the parameters) such that if $\de$ is sufficiently small in terms of all previous parameters, then the lower bound \eqref{proj-to-prove} holds for all $t$ outside of a set of $\wt{\rho}_\lam$-measure $\le (\dess)^{\eta}\le 1/4$.
\end{proof}

\subsection{Conclusion of the proof}
\label{subsec:conclusion-proof}

Fix $\zeta,\eta$ as given by Proposition \ref{prop:application-bourgain-proj-thm}. In this proposition the value of $\lam$ is fixed. We now use a double counting argument to conclude that, always under the standing assumption \eqref{eq:bound-scale-sqr-delta},  $\cN_\de(E)$ must be large.

Let $\cD$ be the grid of axes-parallel squares of side length $\des$. Call a square $D\in\cD$ \emph{rich} if
\[
\cN_\de(E\cap D) \ge c (\dess)^{-\alpha-\beta/2-\zeta},
\]
where $c>0$ is a sufficiently small constant to be chosen later. Let $\cR$ denote the collection of rich squares; our goal is to show that $|\cR|$ is large. For each $\lam\in\Lam'$, let $\cD_\lam$ be all the squares in $\cD$ intersecting the $\des$-neighborhood of $L_\lam$. It follows from Proposition \ref{prop:application-bourgain-proj-thm} (and the Frostman estimate $\wt{\rho}_\lam(B_r)\lesssim r^\alpha$ given in Lemma \ref{lem:Frostman-average}) that, provided $c$ is taken sufficiently small,
\[
|\cR\cap \cD_\lam | \gtrsim (\des)^{-\alpha}.
\]
Hence, recalling Lemma \ref{lem:many-rich-deh-tubes},
\be \label{eq:rich-squares-double-counting}
\sum_{\lam\in\Lam'} |\cR\cap \cD_\lam |\gtrsim (\des)^{-\alpha}|\Lam'| \gtrsim (\des)^{-\alpha-\beta}.
\ee
How many times is each rich square counted? If $\Lambda_D$ denotes the set of all $\lam\in\Lam'$ such that $D\in\cD_\lam$, then $(\nu_\om)_{\om\in\Lambda_D}$ is a $(\des,\alpha,K)$-Furstenberg tuple such that all lines pass at a distance $\lesssim \des$ from the center of $D$. Lemma \ref{lem:fan} and our standing assumption \eqref{eq:bound-scale-sqr-delta} then show that
\[
|\Lambda_D|\lesssim (\des)^{-\beta/2-\eta}.
\]
Hence every rich square contributes to the sum in \eqref{eq:rich-squares-double-counting} at most $C(\des)^{-\beta/2-\eta}$ times, and we conclude that
\[
|\cR|\gtrsim (\des)^{-\alpha-\beta/2+\eta}.
\]
Recalling the definition of rich square,  we conclude that
\[
\cN_\delta(E) \gtrsim \de^{\eta s-\zeta(1-s)} \de^{-\alpha-\beta/2}.
\]
So far, $\eta,\zeta$ must be small in terms of $\al,\beta$ only but are otherwise independent of each other. We can then make $\zeta$ smaller in terms of $\alpha,\beta$ only and choose $\eta$ depending also on $s$ so that $\eta = (1-s)\zeta/4$.

With this choice, and recalling \eqref{eq:bound-scale-sqr-delta} one last time,  we have concluded the proof of \eqref{eq:main-scale-dichotomy} ,with $\zeta/4$ in place of $\zeta$, and with it of Theorem \ref{thm:main-alpha-beta}.


\end{document}